\documentclass[10pt,letterpaper,reqno]{amsart}
\usepackage{amsmath,amsthm,amsfonts,amssymb,mathtools,bbm}
\usepackage[lite]{amsrefs}

\usepackage{bookmark}
\usepackage{hyperref}
\hypersetup{pdfstartview={FitH}}

\theoremstyle{plain}
\newtheorem{theorem}{Theorem}

\newtheorem{lemma}[theorem]{Lemma}
\newtheorem{corollary}[theorem]{Corollary}

\theoremstyle{definition}
\newtheorem{example}[theorem]{Example}

\numberwithin{equation}{section}

\begin{document}

\title[Characterizations of democratic systems of translates]{Characterizations of democratic systems of translates on locally compact abelian groups}

\author{Vjekoslav Kova\v{c}}
\address{Vjekoslav Kova\v{c}, University of Zagreb, Faculty of Science, Department of Mathematics, Bijeni\v{c}ka cesta 30, 10000 Zagreb, Croatia.}
\email{vjekovac@math.hr}

\author{Hrvoje \v{S}iki\'{c}}
\address{Hrvoje \v{S}iki\'{c}, University of Zagreb, Faculty of Science, Department of Mathematics, Bijeni\v{c}ka cesta 30, 10000 Zagreb, Croatia.}
\email{hsikic@math.hr}

\thanks{Both authors are partially supported by the Croatian Science Foundation under the project 3526.}




\begin{abstract}
We present characterizations of democratic property for systems of translates on a general locally compact abelian group, along a lattice in that group. That way we generalize the results from \cite{HNSS13} on systems of integer translates. Furthermore, we investigate the possibilities of more operative characterizations for lattices with torsion group structure, mainly through examples and counterexamples.
\end{abstract}

\maketitle



\section{Introduction}

The study of greedy approximations in Banach spaces brought the democratic property to the attention of
mathematical community. For early results on the greedy algorithm consult the work of V.\,N.\,Temlyakov \cite{Tem1} and \cite{Tem2}.
We would like to emphasize the theorem of S.\,V.\,Konyagin and V.\,N.\,Temlyakov (see \cite{KT99}), which states that a basis in a
Banach space is greedy if and only if it is unconditional and democratic. Recall that conditional bases are often very
difficult to construct and, therefore, the democratic property provides an interesting framework to analyze large
classes of conditional bases. Furthermore, as already outlined in an article by P.\,Wojtaszczyk \cite{Woj}, one can
extend the notion of the democratic property to more general systems than bases. In particular, it is natural to
analyze the democratic property within the realm of various reproducing function systems, like wavelets, Gabor systems, etc.
Within such systems most often we have a core space, which is generated by integer translations of a single
function. This was a point of view taken in \cite{HNSS13}; this paper serves as the main motivating point for
our paper. Let us be more precise here.

A family $\mathcal{F}$ of nonzero vectors in a Banach space $(X,\|\cdot\|)$ is called \emph{democratic} if there exists a
constant $D\in\langle 0,+\infty\rangle$ such that for any two finite subsets $\Gamma_1$ and $\Gamma_2$ of $\mathcal{F}$ with
the same cardinalities one has
\begin{equation}\label{eq:democraticdef}
\Big\|\sum_{f\in\Gamma_1}\frac{f}{\|f\|}\Big\| \leq D \,\Big\|\sum_{f\in\Gamma_2}\frac{f}{\|f\|}\Big\|.
\end{equation}
If all vectors in $\mathcal{F}$ have the same norm, then the above inequality becomes
\[ \Big\|\sum_{f\in\Gamma_1}f\Big\| \leq D \,\Big\|\sum_{f\in\Gamma_2}f\Big\| \]
and it can be rephrased by saying that the norms of the sums $\sum_{f\in\Gamma}f$ are mutually comparable
(up to an absolute constant) for all finite sets $\Gamma\subseteq\mathcal{F}$ of the same size.

Many of the common systems for decomposition and reconstruction of functions begin by considering a closed subspace of $\textup{L}^2(\mathbb{R})$
generated by integer translates of a single square-integrable function. More precisely, one can take $\psi\in\textup{L}^2(\mathbb{R})$, denote
\[ \mathcal{F}_\psi := \big(\psi(\cdot-k)\big)_{k\in\mathbb{Z}},\quad \langle\psi\rangle := \overline{\mathop{\textup{span}}\mathcal{F}_\psi}, \]
and ask to characterize various basis-like properties of the system $\mathcal{F}_\psi$ for the Hilbert space $\langle\psi\rangle$. It turns out that
all such properties can be characterized in terms of the periodization of $|\widehat{\psi}|^2$, defined by
\[ p_\psi(\xi) := \sum_{k\in\mathbb{Z}} |\widehat{\psi}(\xi+k)|^2; \quad \xi\in\mathbb{T}, \]
where $\mathbb{T}=\mathbb{R}/\mathbb{Z}$ is a one-dimensional torus, while the Fourier transform of $f\in\textup{L}^1(\mathbb{R})\cap\textup{L}^2(\mathbb{R})$ is normalized as
\[ \widehat{f}(\xi) := \int_{\mathbb{R}} f(x) e^{-2\pi i x \xi} dx; \quad \xi\in\mathbb{R} \]
and then extended by continuity to the whole $\textup{L}^2(\mathbb{R})$. These characterizations are consequences of the isometric isomorphism between the
shift-invariant space $\langle\psi\rangle$ and the weighted Lebesgue space $\textup{L}^2(\mathbb{T},p_\psi)$, which maps the translates of $\psi$ to the exponentials,
and in turn enables the usage of many classical results from the Fourier analysis. An interested reader can find more details and
numerous references in \cite{HSWW10a}; consult also \cite{CP10}, \cite{HP06}, \cite{HSWW10}, \cite[Sec.~2.1]{HW96},  \cite{HMW73}, \cite{NS07}, \cite{NS14}, \cite{Pa10}, \cite{S13}, \cite{SSl12}, \cite{SSW08}, \cite{Sl16}, and \cite{Sl14}.

Consider the system of translates $\mathcal{F}_\psi$ generated by some $\psi\in\textup{L}^2(\mathbb{R})$. It is natural to try to characterize all $\psi$ for which this system is democratic.
This task was initiated in \cite{HNSS13}. Some characterizing theorems are provided there, as well as numerous necessary or sufficient conditions in various special situations.
However, as stated by the authors in \cite{HNSS13}, the characterizing condition which is at the same time simple and operative remains unknown. It is perhaps somewhat intriguing
why this problem is so difficult, especially in the light of various other basis-like properties that have been successfully treated even on the higher level of generality.
We believe that one aspect of the difficulty of this problem is hidden in the fact that the additive group of integers has an ``unsuitable'' subgroup structure, that is
with respect to this characterization problem. It does not contain any non-trivial elements of the finite order. Why is this important? We propose to elevate the problem
to the higher level of generality and thus reveal the elements of its complexity more clearly. The problem can be formulated for functions on locally compact abelian groups $G$ translated by elements
of a lattice $\mathcal{L}\subseteq G$. As we will see, the democratic property leads to conditions that need to be checked for a very large class of subsets and this problem
is quite demanding from the combinatorics point of view. The problem can be radically simplified, but this discussion leads us to several issues related to seemingly unrelated disciplines,
like the subgroup structure of the group $G$, some ergodic properties and some geometric properties (convexity in particular). In short, we prove that the democratic
property is a condition that is characterized via an operative integrability condition that needs to be tested on a family of extreme points of a certain convex set. This
family is essentially always boxed in between the class of all finite subsets of $G$ (which is usually much larger than the testing family) and the class of all finite subgroups
of $G$ (which, unfortunately, is smaller than the testing family).

Let us give a few words about the organization of the present paper. In Section~\ref{sec:generalconditions} we generalize the results from \cite{HNSS13} at the level of arbitrary locally compact abelian groups. Theorem~\ref{thm:abstractchargamma} characterizes democratic systems of translates by testing sizes of the finite sums \hbox{$\sum_{k\in\Gamma}\psi(\cdot-k)$} for all nonempty finite sets $\Gamma\subseteq\mathcal{L}$. Theorem~\ref{thm:abstractcharppsi} characterizes democratic systems of translates in terms of boundedness and certain density properties of the generalization of the periodization function $p_\psi$. It is suspected that there is a lot of redundancy in verifying those conditions. However, no ``operative'' equivalent conditions are known, even in the case of integer translates, and only a conjecture was stated in the prequel to this paper \cite{HNSS13}. Consequently, Section~\ref{sec:operativeconditions} discusses the possible operative conditions for torsion lattices, but mostly through a series of remarks, examples, counterexamples, and numerical data. We also use the opportunity to formulate a couple of open questions.


\section{Translates along general lattices}
\label{sec:generalconditions}

In this section we first describe the setting in which we formulate the questions mentioned in the introduction. Then we raise the main results from \cite{HNSS13} to the level of abstract abelian groups.


\subsection{Locally compact abelian setting}

Most of the following material is taken from standard books on harmonic analysis on locally compact groups, such as \cite{F95}, \cite{HR79,HR70}, or \cite{R62}, and it has already been adapted to a similar context, for instance in the papers \cite{BR15}, \cite{CP10}, \cite{HSWW10}, \cite{N15}, \cite{SW11}, and \cite{Sl16}.

Whenever we have two subsets $A,B$ of an additively written abelian group $(G,+)$, let us use the notation $A+B$ for the so-called \emph{sumset}, defined as
\[ A+B := \{x+y \,:\, x\in A,\, y\in B\}. \]
If one of the subsets is just a singleton, for instance $A=\{x\}$, then we simply write $x+B$ instead of $\{x\}+B$, and in this case the sumset is just a translate of $B$ by $x$. Analogously we define the \emph{difference set} $A-B$ and the multiple sumset $A_1+A_2+\cdots+A_n$.
The subgroup generated by a set $A\subseteq G$ will be denoted by $\langle A\rangle$; it is the smallest subgroup of $G$ that contains $A$.

Let $(G,+)$ be an abelian topological group, which means that the underlying topology makes both the addition $G\times G\to G$, $(x,y)\mapsto x+y$ and the inversion $G\to G$, $x\mapsto -x$ continuous. Assume that the topology of $G$ is Hausdorff and locally compact, i.e., each point of $G$ has a compact neighborhood. The smallest $\sigma$-algebra on $G$ containing all open subsets is called the \emph{Borel $\sigma$-algebra} and denoted $\mathcal{B}(G)$. There exist a nontrivial Radon measure $\lambda_G$ on $(G,\mathcal{B}(G))$ that is invariant under the translations on $G$, i.e.,
\[ \lambda_G(x+A) = \lambda_G(A) \]
for all $A\in\mathcal{B}(G)$ and $x\in G$. Such measure $\lambda_G$ is unique up to a constant multiple and it is called the \emph{Haar measure} of $G$. It will always be understood whenever we suppress it notationally from integrals or function spaces, i.e., $dx$ will have to be interpreted as $d\lambda_G(x)$ for an appropriate group $G$ and $\textup{L}^p(G)$ will be an abbreviation for $\textup{L}^p(G,\mathcal{B}(G),\lambda_G)$. The measure $\lambda_G$ is finite if and only if $G$ is compact, while in the case of a discrete group $G$ the Haar measure is simply a constant multiple of the counting measure.

If $(\textup{S}^1,\cdot)$ is the group of unimodular complex numbers, then any continuous homomorphism $\xi\colon (G,+)\to (\textup{S}^1,\cdot)$ is called a \emph{(unitary) character} of $G$ and the collection of all characters is denoted $\widehat{G}$. If we endow $\widehat{G}$ with the most obvious pointwise binary operation,
\[ (\xi + \zeta)(x) := \xi(x) \zeta(x); \quad x\in G, \]
and with the topology of uniform convergence on compact subsets of $G$, it also becomes a locally compact Hausdorff topological group. It is called the \emph{dual group} of $G$ and it possesses its own Haar measure $\lambda_{\widehat{G}}$. The group $\widehat{G}$ is compact if and only if $G$ is discrete and vice versa. Let $E_G\colon G\times\widehat{G}\to\textup{S}^1$ be the \emph{bi-character function}, i.e.,
\begin{itemize}
\item $E_G(x,\xi) = \xi(x)$ for $x\in G$ and $\xi\in\widehat{G}$,
\item $E_G(\cdot,\xi)$; $\xi\in\widehat{G}$ are all characters of $G$,
\item $E_G(x,\cdot)$; $x\in G$ are all characters of $\widehat{G}$.
\end{itemize}
Here we implicitly use the Pontrjagin duality theorem, i.e., that the dual group of $\widehat{G}$ is canonically isomorphic to $G$ itself. The fundamental algebraic properties of $E_G$ are
\begin{equation}\label{eq:bicharhom1}
E_G(x+y,\xi) = E_G(x,\xi)E_G(y,\xi),\quad E_G(x,\xi+\zeta) = E_G(x,\xi)E_G(x,\zeta),
\end{equation}
and
\begin{equation}\label{eq:bicharhom2}
E_G(-x,\xi) = \overline{E_G(x,\xi)} = E_G(x,\xi)^{-1} = E_G(x,-\xi)
\end{equation}
for any $x,y\in G$ and any $\xi,\zeta\in\widehat{G}$. The function $E_G$ serves as an analogue of the pure exponentials, since in the classical case $G=\mathbb{R}^d$, $\widehat{G}\cong\mathbb{R}^d$ one can take
\begin{equation}\label{eq:bicharreals}
E_{\mathbb{R}^d}\colon\mathbb{R}^d\times\mathbb{R}^d\to\textup{S}^1,\quad E_{\mathbb{R}^d}(x,\xi) = e^{2\pi i x\cdot\xi},
\end{equation}
where $x\cdot \xi$ stands for the standard inner product on $\mathbb{R}^d$. The \emph{Fourier transform} is initially defined for $f\in\textup{L}^1(G)\cap\textup{L}^2(G)$ by the formula
\begin{equation}\label{eq:fourierdef}
\widehat{f}\colon\widehat{G}\to\mathbb{C},\quad \widehat{f}(\xi) := \int_G f(x) \overline{E_G(x,\xi)} dx;\quad \xi\in\widehat{G}
\end{equation}
and then the isometry $f\mapsto\widehat{f}$ extends to a unitary operator from $\textup{L}^2(G)$ onto $\textup{L}^2(\widehat{G})$ if we choose the appropriate normalization of $\lambda_{\widehat{G}}$ depending on the normalization of $\lambda_G$.

If $H$ is any closed subgroup of $G$, then its \emph{orthogonal complement} is defined as
\[ H^\perp := \{\xi\in\widehat{G} \,:\, E_G(x,\xi)=1 \text{ for each } x\in H\} \]
and it is actually a closed subgroup of $\widehat{G}$. Indeed, $H\mapsto H^\perp$ is a bijective correspondence between closed subgroups of $G$ and closed subgroups of $\widehat{G}$. If $q\colon G\to G/H$ denotes the canonical epimorphism onto the quotient of $G$ by some closed subgroup $H$, then the maps
\begin{equation}\label{eq:groupiso1}
\widehat{G/H} \to H^\perp,\quad \eta\mapsto \eta\circ q
\end{equation}
and
\begin{equation}\label{eq:groupiso2}
\widehat{G}/H^\perp \to \widehat{H},\quad \xi+H^\perp \mapsto \xi|_H
\end{equation}
constitute isomorphisms of topological groups; see \cite[Thm.~4.39]{F95}.

Let us explain how one can construct the bi-character function $E_H$ of $H$ from the bi-character function $E_G$ of $G$. Observe that for any $x\in H$ and any $\xi_1,\xi_2\in\widehat{G}$ such that $\xi_1-\xi_2\in H^\perp$ we have $E_G(x,\xi_1-\xi_2)=1$, so by \eqref{eq:bicharhom1} and \eqref{eq:bicharhom2} we get $E_G(x,\xi_1) = E_G(x,\xi_2)$.
Thus, $E_G\colon G\times\widehat{G}\to\textup{S}^1$ restricted to $H\times\widehat{G}$ factors via a bi-homomorphism $H\times(\widehat{G}/H^\perp)\to\textup{S}^1$, which becomes precisely the bi-character of $H$ after the identification coming from \eqref{eq:groupiso2}.
In other words,
\begin{equation}\label{eq:subgroupbichar}
E_H(x,\zeta) = E_G(x,\xi)
\end{equation}
whenever $x\in H$, $\zeta\in\widehat{H}$, and $\xi\in\widehat{G}$ is any character such that $\xi|_H=\zeta$.

Another structural ingredient we need is a \emph{lattice} $\mathcal{L}$ in $G$, or more precisely, a closed subgroup $\mathcal{L}$ of $G$ such that the relative topology on $\mathcal{L}$ inherited from $G$ becomes discrete and such that the quotient space $G/\mathcal{L}$ is compact (i.e., it is a compact Hausdorff topological group). In all that follows we assume that $\mathcal{L}$ is countably infinite, so that it serves as a reasonable generalization of the integer lattice $\mathbb{Z}^d$. Note that in particular $\mathcal{L}$ cannot be compact, which also forces the whole group $G$ to be noncompact.
Since $G/\mathcal{L}$ is compact, its dual group is discrete, so from \eqref{eq:groupiso1} applied with $H=\mathcal{L}$ we get
\begin{equation}\label{eq:groupiso1l}
\widehat{G/\mathcal{L}}\cong \mathcal{L}^\perp
\end{equation}
and thus we know that $\mathcal{L}^\perp$ is discrete too. On the other hand, since $\mathcal{L}$ is discrete, its dual must be compact, so by applying \eqref{eq:groupiso2} with $H=\mathcal{L}$ we obtain
\begin{equation}\label{eq:groupiso2l}
\widehat{G}/\mathcal{L}^\perp\cong\widehat{\mathcal{L}}
\end{equation}
and we conclude that $\mathcal{L}^\perp$ is cocompact in $\widehat{G}$. From these observations we find that $\mathcal{L}^\perp$ is a lattice in $\widehat{G}$.


\subsection{Testing sizes of finite sums}

The \emph{translate} of $f\in\textup{L}^2(G)$ by $y\in G$ is the function $T_y f\in\textup{L}^2(G)$ defined by
\[ (T_y f)(x) := f(x-y); \quad x\in G. \]
Now we take a function $\psi\in\textup{L}^2(G)$ such that $\|\psi\|_{\textup{L}^2(G)}\neq 0$ and consider its translates by the elements of $\mathcal{L}$,
\[ \mathcal{F}_\psi := \big(T_k \psi\big)_{k\in\mathcal{L}}. \]
The basic question we want to answer is:
\begin{itemize}
\item[] \emph{When does $\mathcal{F}_\psi$ constitute a democratic system in $\textup{L}^2(G)$ in the sense of the general definition \eqref{eq:democraticdef}?}
\end{itemize}
The first step towards the resolution is the following straightforward generalization of \cite[Thm.~4.3]{HNSS13}. Recall that we have assumed $\mathcal{L}$ to be infinite and countable. The characterization below fails when $\mathcal{L}$ is finite, but finite democratic families are not particularly interesting and they have already been discussed in \cite{HNSS13}.

\begin{theorem}\label{thm:abstractchargamma}
The system $\mathcal{F}_\psi$ is democratic if and only if there exist constants $c_\psi,C_\psi\in\langle 0,+\infty\rangle$ such that for every nonempty finite set $\Gamma\subset\mathcal{L}$ we have
\begin{equation}\label{eq:basicchargamma}
c_\psi\mathop{\textup{card}}\Gamma \leq \Big\|\sum_{k\in\Gamma} T_k\psi\Big\|_{\textup{L}^2(G)}^2 \leq C_\psi\mathop{\textup{card}}\Gamma.
\end{equation}
\end{theorem}

\begin{proof}[Proof of Theorem~\ref{thm:abstractchargamma}]
The sufficiency is trivial, as one can simply take $D=\sqrt{C_\psi/c_\psi}$.
For the necessity it is enough to show that for each positive integer $n$ there exists a subset $\Gamma_n\subset\mathcal{L}$ of cardinality $n$ such that
\begin{equation}\label{eq:proofinductest}
\frac{1}{2}n\|\psi\|_{\textup{L}^2(G)}^2 \leq \Big\|\sum_{k\in\Gamma_n} T_k\psi\Big\|_{\textup{L}^2(G)}^2 \leq \frac{3}{2}n\|\psi\|_{\textup{L}^2(G)}^2.
\end{equation}
This statement is easily established by the induction on $n$. For the induction basis $n=1$ we can simply take $\Gamma_1$ to be a singleton containing an arbitrary element of $\mathcal{L}$ and use the translation invariance of the $\textup{L}^2$ norm. In the induction step we take $n\geq 2$ and assume that we have already constructed the set $\Gamma_{n-1}$. By a density argument one can first find a compactly supported continuous function $\varphi$ on $G$ such that $\|\varphi-\psi\|_{\textup{L}^2(G)} < (1/10n^2)\|\psi\|_{\textup{L}^2(G)}$. Denote $K := \bigcup_{k\in\Gamma_{n-1}} (k+\mathop{\textup{supp}}\varphi)$. Since the set $K-\mathop{\textup{supp}}\varphi$ is compact, while $\mathcal{L}$ is not, there exists at least one element $m\in\mathcal{L}\setminus(K-\mathop{\textup{supp}}\varphi)$. Thus $m+\mathop{\textup{supp}}\varphi$ is disjoint from $K$, i.e., $\mathop{\textup{supp}}T_m\varphi$ and $\bigcup_{k\in\Gamma_{n-1}}\mathop{\textup{supp}}T_k\varphi$ are mutually disjoint sets. If we define $\Gamma_n := \Gamma_{n-1}\cup\{m\}$, we will get a subset of cardinality $n$ and the previously mentioned disjointness of supports will yield
\begin{align*}
& \Big\|\sum_{k\in\Gamma_n} T_k\varphi \Big\|_{\textup{L}^2(G)}^2
= \Big\|\sum_{k\in\Gamma_{n-1}} T_k\varphi + T_m\varphi \Big\|_{\textup{L}^2(G)}^2 \\
& = \Big\|\sum_{k\in\Gamma_{n-1}} T_k\varphi\Big\|_{\textup{L}^2(G)}^2 + \|T_m\varphi\|_{\textup{L}^2(G)}^2
= \Big\|\sum_{k\in\Gamma_{n-1}} T_k\varphi\Big\|_{\textup{L}^2(G)}^2 + \|\varphi\|_{\textup{L}^2(G)}^2.
\end{align*}
Using this equality and estimating
\[ \big| \|f\|_{\textup{L}^2(G)}^2 - \|g\|_{\textup{L}^2(G)}^2 \big| \leq \|f-g\|_{\textup{L}^2(G)} \big(\|f\|_{\textup{L}^2(G)}+\|g\|_{\textup{L}^2(G)}\big) \]
for any $f,g\in\textup{L}^2(G)$, we see that
\[ \Big\|\sum_{k\in\Gamma_n} T_k\psi \Big\|_{\textup{L}^2(G)}^2 \quad\text{and}\quad \Big\|\sum_{k\in\Gamma_{n-1}} T_k\psi\Big\|_{\textup{L}^2(G)}^2 + \|\psi\|_{\textup{L}^2(G)}^2 \]
differ by at most
\[ \big(n^2+(n-1)^2+1\big) \|\varphi-\psi\|_{\textup{L}^2(G)} \big(\|\varphi\|_{\textup{L}^2(G)}+\|\psi\|_{\textup{L}^2(G)}\big), \]
which is, by the choice of $\varphi$, less than $\|\psi\|_{\textup{L}^2(G)}^2$ times
\[ \big(n^2+(n-1)^2+1\big) \frac{1}{10n^2} \Big(2+\frac{1}{10n^2}\Big) < \frac{1}{2}. \]
Combining this with the induction hypothesis
\[ \frac{1}{2}(n-1)\|\psi\|_{\textup{L}^2(G)}^2 \leq \Big\|\sum_{k\in\Gamma_{n-1}} T_k\psi\Big\|_{\textup{L}^2(G)}^2 \leq \frac{3}{2}(n-1)\|\psi\|_{\textup{L}^2(G)}^2 \]
gives \eqref{eq:proofinductest} and completes the inductive proof.
\end{proof}

Characterization from Theorem~\ref{thm:abstractchargamma} is still not especially operative, since it might be redundant to verify condition \eqref{eq:basicchargamma} for all finite subsets $\Gamma$.


\subsection{Testing the periodization function}

One can also hope to phrase the answer to the main question in terms of the \emph{periodization function}, which is now defined as
\[ p_\psi(\xi) := \sum_{m\in\mathcal{L}^\perp} |\widehat{\psi}(\xi+m)|^2; \quad \xi\in\widehat{G}. \]
Since $p_\psi(\xi)$ only depends on the coset of $\widehat{G}/\mathcal{L}^\perp$ to which $\xi$ belongs, by \eqref{eq:groupiso2l} we can view $p_\psi$ as a function $\widehat{\mathcal{L}}\to[0,+\infty]$, $\tau\mapsto p_\psi(\tau):=p_\psi(\xi)$. Here one has to extend the character $\tau$ of $\mathcal{L}$ arbitrarily to a character $\xi$ of $G$, which is always possible via the isomorphism \eqref{eq:groupiso2l}. Moreover, the formula in \cite[Thm.~2.49]{F95} tells us how to integrate over the quotient group, so we have
\begin{align*}
& \int_{\widehat{G}/\mathcal{L}^\perp} p_\psi(\xi) d\xi = \int_{\widehat{G}/\mathcal{L}^\perp} \!\Big(\sum_{m\in\mathcal{L}^\perp} |\widehat{\psi}(\xi+m)|^2\Big) d\xi \\
& = \int_{\widehat{G}} |\widehat{\psi}(\zeta)|^2 d\zeta = \|\widehat{\psi}\|_{\textup{L}^2(\widehat{G})}^2 = \|\psi\|_{\textup{L}^2(G)}^2 < +\infty.
\end{align*}
Hence, in fact $p_\psi\in\textup{L}^{1}(\widehat{G}/\mathcal{L}^\perp)\cong\textup{L}^{1}(\widehat{\mathcal{L}})$ and in particular $p_\psi$ is finite a.e.

The bi-character of $\mathcal{L}$ will be denoted $e\colon\mathcal{L}\times\widehat{\mathcal{L}}\to\textup{S}^1$ from now on. The one thing we need to remember from the construction leading to \eqref{eq:subgroupbichar} is that
\begin{equation}\label{eq:smallandbige}
e(k,\tau) = E_G(k,\xi)
\end{equation}
whenever $k\in\mathcal{L}$, $\xi\in\widehat{G}$, and $\tau=\xi|_\mathcal{L}\in\widehat{\mathcal{L}}$. For example, when $G=\mathbb{R}^d$, $\mathcal{L}=\mathbb{Z}^d$, then $\widehat{\mathcal{L}}\cong\mathbb{T}^d$ and in addition to \eqref{eq:bicharreals} we have
\[ e\colon\mathbb{Z}^d\times\mathbb{T}^d\to\textup{S}^1,\quad e(k,\tau) = e^{2\pi i k\cdot\tau}. \]
The Haar measure on $\widehat{\mathcal{L}}$ is finite because this group is compact. We find convenient to normalize it so that $\lambda_{\widehat{\mathcal{L}}}(\widehat{\mathcal{L}})=1$.

Let us agree to denote
\begin{equation}\label{eq:gfunctdef}
g_\Gamma\colon\widehat{\mathcal{L}}\to[0,+\infty\rangle,\quad g_\Gamma(\tau) := \frac{1}{\mathop{\textup{card}}\Gamma}\Big|\sum_{k\in\Gamma}e(k,\tau)\Big|^2; \quad \tau\in\widehat{\mathcal{L}}
\end{equation}
for every nonempty finite set $\Gamma\subset\mathcal{L}$. By \eqref{eq:groupiso2l} and \eqref{eq:smallandbige} we can also interpret $g_\Gamma$ as a function on $\widehat{G}/\mathcal{L}^\perp$. From \eqref{eq:bicharhom1} we immediately see that translating $\Gamma$ does not affect $g_\Gamma$. Basic properties of the functions $g_\Gamma$ are given in the following lemma and its part (b) also motivates their definition.


\begin{lemma}\label{lemma:gdefprop}
Let $\Gamma\subset\mathcal{L}$ be an arbitrary nonempty finite set.
\begin{itemize}

\item[(a)]
The function $g_\Gamma$ defined by \eqref{eq:gfunctdef} is nonnegative, even, continuous, and it integrates to $1$.
Moreover, $g_\Gamma$ is a generalized trigonometric polynomial (i.e., a finite linear combination of characters $e(k,\cdot)$; $k\in\mathcal{L}$) with all coefficients in $[0,1]$, and its expansion is given by
\[ g_\Gamma(\tau) = \sum_{k\in\mathcal{L}} \frac{\mathop{\textup{card}}(\Gamma\cap(k+\Gamma))}{\mathop{\textup{card}}\Gamma}\, e(k,\tau); \quad \tau\in\widehat{\mathcal{L}}. \]

\item[(b)]
For any $\psi\in\textup{L}^2(G)$ we have
\[ \frac{1}{\mathop{\textup{card}}\Gamma}\Big\|\sum_{k\in\Gamma} T_k\psi\Big\|_{\textup{L}^2(G)}^2
= \int_{\widehat{\mathcal{L}}} g_\Gamma(\tau) p_\psi(\tau) d\tau. \]

\end{itemize}
\end{lemma}

\begin{proof}[Proof of Lemma~\ref{lemma:gdefprop}]
(a) The first three mentioned properties of $g_\Gamma$ are obvious. We can use the well-known fact that the characters of $\widehat{\mathcal{L}}$ form an orthonormal basis for $\textup{L}^2(\widehat{\mathcal{L}})$ to also conclude
\[ \int_{\widehat{\mathcal{L}}} g_\Gamma(\tau) d\tau = \frac{1}{\mathop{\textup{card}}\Gamma} \Big\|\sum_{k\in\Gamma}e(k,\cdot)\Big\|_{\textup{L}^2(\widehat{\mathcal{L}})}^2 = \frac{1}{\mathop{\textup{card}}\Gamma} \sum_{k\in\Gamma}1^2 = 1. \]
Furthermore, expanding the expression from definition \eqref{eq:gfunctdef} and using \eqref{eq:bicharhom1} and \eqref{eq:bicharhom2} we get
\begin{align*}
g_\Gamma(\tau) & = \frac{1}{\mathop{\textup{card}}\Gamma} \sum_{k_1,k_2\in\Gamma} e(k_1,\tau)\overline{e(k_2,\tau)}
= \frac{1}{\mathop{\textup{card}}\Gamma} \sum_{k_1,k_2\in\Gamma} e(k_1-k_2,\tau) \\
& = \frac{1}{\mathop{\textup{card}}\Gamma} \,\sum_{k\in\mathcal{L}} \Big(\sum_{\substack{k_1,k_2\in\Gamma\\k_1-k_2=k}}1\Big)\, e(k,\tau)
= \sum_{k\in\mathcal{L}} \frac{\mathop{\textup{card}}(\Gamma\cap(k+\Gamma))}{\mathop{\textup{card}}\Gamma}\, e(k,\tau).
\end{align*}
Observe that the summation in $k$ is actually taken over a finite set $\Gamma-\Gamma\subset\mathcal{L}$, so $g_\Gamma$ really is a generalized trigonometric polynomial.

(b) It is readily verified that
\[ (\widehat{T_y f})(\xi) := \overline{E_G(y,\xi)} \widehat{f}(\xi); \quad \xi\in\widehat{G} \]
for $f\in\textup{L}^2(G)$ by an easy application of formula \eqref{eq:fourierdef} on a dense subspace $\textup{L}^1(G)\cap\textup{L}^2(G)$. This fact combined with unitarity of the Fourier transform enables us the following computation for a nonempty finite subset $\Gamma$ of $\mathcal{L}$:
\begin{align*}
& \Big\|\sum_{k\in\Gamma} T_k\psi\Big\|_{\textup{L}^2(G)}^2 = \Big\|\sum_{k\in\Gamma} \widehat{T_k\psi}\Big\|_{\textup{L}^2(\widehat{G})}^2
= \int_{\widehat{G}} \underbrace{\Big|\sum_{k\in\Gamma}E_G(k,\xi)\Big|^2}_{\text{$\mathcal{L}^\perp$-periodic}} |\widehat{\psi}(\xi)|^2 d\xi \\
& = \int_{\widehat{G}/\mathcal{L}^\perp} \Big|\sum_{k\in\Gamma}E_G(k,\xi)\Big|^2 \Big(\sum_{m\in\mathcal{L}^\perp}|\widehat{\psi}(\xi+m)|^2\Big) d\xi
= \int_{\widehat{\mathcal{L}}} \Big|\sum_{k\in\Gamma}e(k,\tau)\Big|^2 p_\psi(\tau) d\tau.
\end{align*}
It remains to divide by $\mathop{\textup{card}}\Gamma$ and apply the definition of $g_\Gamma$.
\end{proof}

Now we can characterize democracy in terms of the periodization function $p_\psi$.

\begin{theorem}\label{thm:abstractcharppsi}
The system $\mathcal{F}_\psi$ is democratic if and only if $p_\psi$ is essentially bounded (i.e., $p_\psi\in\textup{L}^\infty(\widehat{\mathcal{L}})$) and
\begin{equation}\label{eq:basiccharppsi}
\inf_{\Gamma} \int_{\widehat{\mathcal{L}}} g_\Gamma(\tau) p_\psi(\tau) d\tau > 0,
\end{equation}
where the infimum is taken over all nonempty finite sets $\Gamma\subset\mathcal{L}$.
\end{theorem}

Before the proof we will establish two auxiliary result. As the first ingredient we need an appropriate analogue of the classical Dirichlet kernel on $\mathbb{T}$, which was used in \cite[Thm.~4.7]{HNSS13} for the same purpose, and the classical Fej\'{e}r kernel derived from it.

The key concept that we borrow from ergodic theory is a \emph{F{\o}lner sequence} \cite{F55} for $\mathcal{L}$, which is a sequence $(F_n)_{n=1}^{\infty}$ of nonempty finite subsets of $\mathcal{L}$ such that for each $k\in\mathcal{L}$ one has
\begin{equation}\label{eq:folnerproperty}
\lim_{n\to\infty} \frac{\mathop{\textup{card}}(F_n\triangle(k+F_n))}{\mathop{\textup{card}}F_n} = 0.
\end{equation}
Here $\triangle$ denotes the symmetric difference of sets, i.e., $A\triangle B:=(A\setminus B)\cup(B\setminus A)$.
It is a well-known fact that every countable discrete abelian group possesses a F{\o}lner sequence, which is just one of many equivalent ways of saying that each countable discrete abelian group is \emph{amenable}; see \cite{P88}.
Indeed, an obvious choice of the F{\o}lner sets for $\mathbb{Z}$ are discrete intervals $F_n=\{-N,\ldots,N\}$. One similarly verifies that every finitely generated abelian group is amenable. Finally, while considering an arbitrary (not necessarily finitely generated) countable discrete abelian group one only needs to observe that the direct limit of an increasing sequence of countable amenable groups is also amenable; see \cite{T09} for details.

Furthermore, recall that $\textup{L}^1(\widehat{\mathcal{L}})$ is a Banach algebra with respect to the convolution as multiplication, which is in turn defined as
\[ (f\ast h)(\tau) := \int_{\widehat{\mathcal{L}}} f(\tau-\sigma) h(\sigma) d\sigma. \]
This algebra does not have an identity. We say that a sequence $(f_n)_{n=1}^{\infty}$ is an \emph{approximate identity} (a notion used for instance in \cite[Sec.~2.5]{F95} and \cite[Sec.~V.20]{HR79}) for $\textup{L}^1(\widehat{\mathcal{L}})$ if for every function $h$ in that space we have $\lim_{n\to\infty} f_n\ast h = h$ with convergence in the $\textup{L}^1$ norm.

\begin{lemma}\label{lemma:folnersequence}
If $(F_n)_{n=1}^{\infty}$ is a F{\o}lner sequence for $\mathcal{L}$, then $(g_{F_n})_{n=1}^{\infty}$ constitutes an approximate identity for $\textup{L}^1(\widehat{\mathcal{L}})$.
\end{lemma}

\begin{proof}[Proof of Lemma~\ref{lemma:folnersequence}]
Let us begin by taking a generalized trigonometric polynomial
\[ q(\tau) = \sum_{k\in\Gamma} \alpha_k e(k,\tau) \]
for a finite set $\Gamma\subset\mathcal{L}$ and some complex coefficients $(\alpha_k)_{k\in\Gamma}$.
Observe that by part (a) of Lemma~\ref{lemma:gdefprop} we have
\[ (g_{F_n}\!\ast q)(\tau) = \sum_{k\in\Gamma} \frac{\mathop{\textup{card}}(F_n\cap(k+F_n))}{\mathop{\textup{card}}F_n} \,\alpha_k e(k,\tau) \]
and thus also
\[ q(\tau) - (g_{F_n}\!\ast q)(\tau) = \sum_{k\in\Gamma} \frac{\mathop{\textup{card}}(F_n\setminus(k+F_n))}{\mathop{\textup{card}}F_n} \,\alpha_k e(k,\tau), \]
so we can estimate
\begin{equation}\label{eq:folnerauxest}
\| g_{F_n}\!\ast q - q \|_{\textup{L}^1(\widehat{\mathcal{L}})} \leq \sum_{k\in\Gamma}
\frac{\mathop{\textup{card}}(F_n\setminus(k+F_n))}{\mathop{\textup{card}}F_n} |\alpha_k|.
\end{equation}
For each $k\in\Gamma$ the corresponding term in \eqref{eq:folnerauxest} converges to $0$ by the F{\o}lner property \eqref{eq:folnerproperty}, which implies
\begin{equation}\label{eq:folnerauxest2}
\lim_{n\to\infty} \| g_{F_n}\!\ast q - q \|_{\textup{L}^1(\widehat{\mathcal{L}})} = 0.
\end{equation}

Now take an arbitrary $h\in\textup{L}^1(\widehat{\mathcal{L}})$ and an $\varepsilon>0$. By density there exists a generalized trigonometric polynomial $q$ such that $\|q-h\|_{\textup{L}^1(\widehat{\mathcal{L}})}<\varepsilon/3$, while by \eqref{eq:folnerauxest2} there exists a positive integer $n_0$ such that $n\geq n_0$ implies $\|g_{F_n}\!\ast q-q\|_{\textup{L}^1(\widehat{\mathcal{L}})}<\varepsilon/3$. Therefore, for each index $n\geq n_0$ we have
\[ \| g_{F_n}\!\ast h - h \|_{\textup{L}^1(\widehat{\mathcal{L}})}
\leq \| g_{F_n} \|_{\textup{L}^1(\widehat{\mathcal{L}})} \| h - q \|_{\textup{L}^1(\widehat{\mathcal{L}})}
+ \| g_{F_n}\!\ast q - q \|_{\textup{L}^1(\widehat{\mathcal{L}})}
+ \| q - h \|_{\textup{L}^1(\widehat{\mathcal{L}})}
< \varepsilon. \qedhere \]
\end{proof}

Existence of approximate identities as in Lemma~\ref{lemma:folnersequence} at the level of general countable abelian lattices was named the \emph{Fej\'{e}r property} in \cite{N15}. It was the working assumption in that paper.

The following lemma is a certain folklore and it appears (in some form and with a larger constant) in many texts on Banach spaces; for instance see \cite[Sec.~3.4]{H11}. For completeness we give its elegant self-contained proof.

\begin{lemma}\label{lemma:banachselection}
For any finite system of vectors $(f_k)_{k\in\Gamma}$ in a complex Banach space $(X,\|\cdot\|)$ and any finite system of coefficients $(\alpha_k)_{k\in\Gamma}$ satisfying $|\alpha_k|\leq 1$ for each $k\in\Gamma$, we have
\[ \Big\|\sum_{k\in\Gamma} \alpha_k f_k\Big\| \leq \pi \max_{\Gamma'\subseteq\Gamma} \Big\|\sum_{k\in\Gamma'} f_k\Big\|, \]
where the maximum is taken over all subsets $\Gamma'$ of $\Gamma$.
\end{lemma}

\begin{proof}[Proof of Lemma~\ref{lemma:banachselection}]
By the dual characterization of the norm,
\[ \|f\| = \sup_{\varphi\in X^\ast,\,\|\varphi\|\leq1} |\varphi(f)|, \]
it is enough to show that for each continuous linear functional $\varphi$ one has
\begin{equation}\label{eq:absfunctional}
\sum_{k\in\Gamma} |\varphi(f_k)| \leq \pi \max_{\Gamma'\subseteq\Gamma} \Big|\varphi\Big(\sum_{k\in\Gamma'}f_k\Big)\Big|.
\end{equation}
Indeed, by the assumption on the coefficients $\alpha_k$ we can then estimate
\begin{align*}
\Big\|\sum_{k\in\Gamma} \alpha_k f_k\Big\|
& = \sup_{\varphi\in X^\ast,\,\|\varphi\|\leq1} \Big|\sum_{k\in\Gamma} \alpha_k \varphi(f_k)\Big|
\leq \sup_{\varphi\in X^\ast,\,\|\varphi\|\leq1} \sum_{k\in\Gamma} |\varphi(f_k)| \\
& \leq \pi \sup_{\varphi\in X^\ast,\,\|\varphi\|\leq1} \max_{\Gamma'\subseteq\Gamma} \Big|\varphi\Big(\sum_{k\in\Gamma'}f_k\Big)\Big|
\leq \pi \max_{\Gamma'\subseteq\Gamma} \Big\|\sum_{k\in\Gamma'} f_k\Big\|.
\end{align*}

In order to show \eqref{eq:absfunctional}, observe that for every $\beta\in\mathbb{C}$ we have the identity
\[ \int_{\mathbb{T}} \max\{ \textup{Re}(e^{2\pi i\theta}\beta), 0 \} d\theta
= \int_{-1/4}^{1/4} |\beta|\cos(2\pi\theta) d\theta = \frac{|\beta|}{\pi}. \]
Therefore, if for each $\theta\in\mathbb{T}$ we define
\[ \Gamma'_{\theta} := \big\{k\in\Gamma \,:\, \textup{Re}\big(e^{2\pi i\theta}\varphi(f_k)\big) \geq 0\big\}, \]
then
\begin{align*}
\sum_{k\in\Gamma} |\varphi(f_k)|
& = \pi \int_{\mathbb{T}} \sum_{k\in\Gamma'_{\theta}} \textup{Re}\big(e^{2\pi i\theta}\varphi(f_k)\big) d\theta
= \pi \int_{\mathbb{T}} \textup{Re}\Big(e^{2\pi i\theta} \varphi\Big(\sum_{k\in\Gamma'_{\theta}}f_k\Big) \Big) d\theta \\
& \leq \pi \int_{\mathbb{T}} \Big| \varphi\Big(\sum_{k\in\Gamma'_{\theta}}f_k\Big) \Big| d\theta
\leq \pi \max_{\Gamma'\subseteq\Gamma} \Big|\varphi\Big(\sum_{k\in\Gamma'}f_k\Big)\Big|. \qedhere
\end{align*}
\end{proof}

It is interesting to observe that the constant $\pi$ in Lemma~\ref{lemma:banachselection} is optimal, already in the one-dimensional case $X=\mathbb{C}$.
Indeed, for a positive integer $n$ take $\Gamma:=\{0,1,\ldots,2n-1\}$, $f_k:=e^{\pi i k/n}$, and $\alpha_k:=e^{-\pi i k/n}$.
It can be easily seen that the largest absolute value of the sum over a subset equals $|\sum_{k=0}^{n-1}f_k|=1/\sin(\pi/2n)$ and its ratio to the left hand side $|\sum_{k=0}^{2n-1} \alpha_k f_k|=2n$ converges to $1/\pi$ as $n\to\infty$.

Finally, we are ready to give the proof of the desired characterization of democracy, following the outline from \cite{HNSS13}.

\begin{proof}[Proof of Theorem~\ref{thm:abstractcharppsi}]
We will reduce the claim to Theorem~\ref{thm:abstractchargamma}. By part (b) of Lemma~\ref{lemma:gdefprop} we immediately see that the left inequality in \eqref{eq:basicchargamma} is equivalent to condition \eqref{eq:basiccharppsi}, so it remains to show that the right inequality is equivalent with essential boundedness of $p_\psi$.

One implication is trivial, as $p_\psi\in\textup{L}^\infty(\widehat{\mathcal{L}})$ together with part (b) of Lemma~\ref{lemma:gdefprop} guarantees
\[ \frac{1}{\mathop{\textup{card}}\Gamma} \Big\|\sum_{k\in\Gamma} T_k\psi\Big\|_{\textup{L}^2(G)}^2
\leq \|g_\Gamma\|_{\textup{L}^1(\widehat{\mathcal{L}})} \|p_\psi\|_{\textup{L}^\infty(\widehat{\mathcal{L}})}
= \|p_\psi\|_{\textup{L}^\infty(\widehat{\mathcal{L}})} < +\infty \]
for every finite subset $\Gamma$ of $\mathcal{L}$.

Conversely, suppose that $\psi$ is such that the right inequality in \eqref{eq:basicchargamma} holds with some finite constant $C_\psi$. By Lemma~\ref{lemma:banachselection} for any finite set $\Gamma\subset\mathcal{L}$ and any coefficients $(\alpha_k)_{k\in\Gamma}$ satisfying $|\alpha_k|\leq 1$ we also have
\[ \Big\|\sum_{k\in\Gamma} \alpha_k T_k\psi\Big\|_{\textup{L}^2(G)}^2 \leq C\mathop{\textup{card}}\Gamma, \]
where $C:=\pi^2 C_\psi$.
The same computation from the proof of Lemma~\ref{lemma:gdefprop} now gives
\[ \int_{\widehat{\mathcal{L}}} \frac{1}{\mathop{\textup{card}}\Gamma} \Big|\sum_{k\in\Gamma}\alpha_k\overline{e(k,\tau)}\Big|^2 p_\psi(\tau) d\tau \leq C \]
and the particular choice $\alpha_k = e(k,\sigma)$ for a fixed $\sigma\in\widehat{\mathcal{L}}$ simplifies to
\begin{equation}\label{eq:proppaux}
(g_\Gamma \ast p_\psi)(\sigma) = \int_{\widehat{\mathcal{L}}} g_\Gamma(\sigma-\tau) p_\psi(\tau) d\tau \leq C.
\end{equation}
Finally, we take a F{\o}lner sequence $(F_n)_{n=1}^{\infty}$ for $\mathcal{L}$ and recall that the sequence $(g_{F_n} \ast p_\psi)_{n=1}^{\infty}$ converges to $p_\psi$ in the $\textup{L}^1$ norm by Lemma~\ref{lemma:folnersequence}. There exist a subsequence $(g_{F_{n_j}} \ast p_\psi)_{j=1}^{\infty}$ that converges a.e., so by taking $\Gamma=F_{n_j}$ in \eqref{eq:proppaux} and letting $j\to\infty$ we conclude that $p_\psi \leq C$ a.e.
\end{proof}

Theorem~\ref{thm:abstractcharppsi} is still not much more practical than Theorem~\ref{thm:abstractchargamma}, as condition \eqref{eq:basiccharppsi} requires computation of a certain integral for each finite subset $\Gamma$. However, there is one thing worth noticing that has now become evident: the democratic property of the system $\mathcal{F}_\psi$ is characterized only in terms of the lattice and the periodization function. More precisely, once we are given the function $p_\psi$ on $\widehat{\mathcal{L}}$ (rather than $\widehat{G}/\mathcal{L}^\perp$), the democracy condition depends only on properties of the lattice $\mathcal{L}$ (and its bi-character $e$), but not on the ambient group $G$. In particular, this means that the general case is not any more difficult than the special case $G=\mathcal{L}$. In this latter case $\mathcal{L}^\perp$ is trivial by \eqref{eq:groupiso1l} and Theorem~\ref{thm:abstractcharppsi} calls for a characterization in terms of $p_\psi=|\widehat{\psi}|^2$, which can a priori be an arbitrary nonnegative integrable function on $\widehat{G}=\widehat{\mathcal{L}}$.

Therefore, in the remaining text we mostly work on the lattice $(\mathcal{L},+)$ and its (compact) dual group $(\widehat{\mathcal{L}},+)$. For any $A\in\mathcal{B}(\widehat{\mathcal{L}})$ we will simply write its measure as $|A|$, instead of $\lambda_{\widehat{\mathcal{L}}}(A)$, in analogy with the usual practice on the torus $\mathbb{T}$. Recall that we have chosen the normalization so that $|\widehat{\mathcal{L}}|=1$.

We end this section with a sufficient condition from democracy, which is just an adaptation of \cite[Cor.~4.20]{HNSS13}.

\begin{corollary}\label{cor:gensufficient}
If there exist constants $c,C\in\langle0,+\infty\rangle$ and an open neighborhood $U$ of $0$ in $\widehat{\mathcal{L}}$ such that
\[ p_\psi(\tau) \leq C \ \text{for a.e.}\ \tau\in\widehat{\mathcal{L}} \quad\text{and}\quad p_\psi(\tau) \geq c \ \text{for a.e.}\ \tau\in U, \]
then $\mathcal{F}_\psi$ is a democratic system.
\end{corollary}

\begin{proof}[Proof of Corollary~\ref{cor:gensufficient}]
This will be a consequence of Theorem~\ref{thm:abstractcharppsi} as soon as we show that the hypothesis $p_\psi \geq c$ a.e.\@ on $U$ implies \eqref{eq:basiccharppsi}.
By \cite[Prop.~2.1]{F95} or \cite[Thm.~4.5\&4.6]{HR79} there exists an open neighborhood $V\subseteq\widehat{\mathcal{L}}$ of $0$ such that $V+V\subseteq U$ and $-V=V$.
Its characteristic function $\mathbbm{1}_V$ vanishes outside $U$ and satisfies $0\leq(\mathbbm{1}_V\ast\mathbbm{1}_V)(\tau)\leq |V|$ for each $\tau\in\widehat{\mathcal{L}}$.
From the properties of the Haar measure we know $|V|>0$; see \cite[Lm.~9.2.5]{C13} or \cite[Rem.~15.8]{HR79}.
Finally, for each nonempty finite $\Gamma\subset\mathcal{L}$ by part (a) of Lemma~\ref{lemma:gdefprop} we have
{\allowdisplaybreaks\begin{align*}
& \int_{\widehat{\mathcal{L}}} g_\Gamma(\tau) p_\psi(\tau) d\tau \geq c \int_{U} g_\Gamma(\tau) d\tau
\geq \frac{c}{|V|}\int_{\widehat{\mathcal{L}}} g_\Gamma(\tau) (\mathbbm{1}_V\ast\mathbbm{1}_V)(\tau) d\tau \\
& = \frac{c}{|V|}\sum_{k\in\Gamma-\Gamma} \frac{\mathop{\textup{card}}(\Gamma\cap(k+\Gamma))}{\mathop{\textup{card}}\Gamma} \int_{\widehat{\mathcal{L}}}\int_{\widehat{\mathcal{L}}} e(k,\sigma) \overline{e(k,\sigma-\tau)} \mathbbm{1}_{V}(\sigma) \mathbbm{1}_{V}(\sigma-\tau) d\sigma d\tau \\
& = \frac{c}{|V|}\sum_{k\in\Gamma-\Gamma} \frac{\mathop{\textup{card}}(\Gamma\cap(k+\Gamma))}{\mathop{\textup{card}}\Gamma} \Big| \int_{\widehat{\mathcal{L}}} e(k,\sigma) \mathbbm{1}_{V}(\sigma) d\sigma \Big|^2
\geq \frac{c}{|V|} |V|^2 = c|V|,
\end{align*}}
where in the last inequality we estimated the sum of nonnegative terms by its single term for $k=0$.
Therefore,
\[ \inf_\Gamma \int_{\widehat{\mathcal{L}}} g_\Gamma(\tau) p_\psi(\tau) d\tau \geq c|V| >0, \]
as desired.
\end{proof}

Even though we would like to find more operative equivalent conditions for democratic systems of translates, obtaining these is an open problem even in the particular case when the lattice is $\mathcal{L}=\mathbb{Z}$, as commented in the paper \cite{HNSS13}.
It was conjectured in \cite{HNSS13} that it is sufficient to test the democratic property on the arithmetic progressions
\[ \Gamma = \{0,d,2d,\ldots,(m-1)d\} \]
for positive integers $d$ and $m$. In the same paper it was also conjectured that the system $\mathcal{F}_\psi$ is democratic if and only if $p_\psi\in\textup{L}^\infty(\mathbb{T})$ and
\[ \inf_{\substack{d\in\mathbb{N}\\ 0<\varepsilon<1/2}} \frac{1}{2\varepsilon}\int_{\bigcup_{j=0}^{d-1}\left[\frac{j-\varepsilon}{d},\frac{j+\varepsilon}{d}\right]} p_\psi(\tau) d\tau > 0. \]


\section{Translates along torsion lattices}
\label{sec:operativeconditions}

The goal of this section is to investigate further the conditions from Theorems~\ref{thm:abstractchargamma} and \ref{thm:abstractcharppsi} from the viewpoint of convex geometry.

Some lattices $\mathcal{L}$ have very different algebraic structure from the integers. We say that $\mathcal{L}$ is a \emph{torsion lattice} if $(\mathcal{L},+)$ is also a torsion group, i.e., each of its elements has finite order.
This is clearly equivalent to the fact that there exists a sequence $(\mathcal{M}_n)_{n=0}^{\infty}$ of finite subgroups of $\mathcal{L}$ such that
\begin{equation}\label{eq:sequencesubgroups}
\{0\} = \mathcal{M}_0 \subset \mathcal{M}_1 \subset \mathcal{M}_2 \subset\cdots \quad\text{and}\quad
\bigcup_{n=0}^{\infty} \mathcal{M}_n = \mathcal{L}.
\end{equation}
In order to construct such a sequence recursively one has to enumerate the lattice as $\mathcal{L}=\{k_i\}_{i=1}^{\infty}$ and at the $n$-th step take the smallest index $i$ such that $k_i\not\in\mathcal{M}_{n-1}$, observing that the sumset $\mathcal{M}_n:=\mathcal{M}_{n-1}+\langle\{k_i\}\rangle$ is a finite subgroup.
Any such sequence of subgroups is also a F{\o}lner sequence, i.e., it satisfies the F{\o}lner property \eqref{eq:folnerproperty}, which is a trivial consequence of the fact that $k+\mathcal{M}_n=\mathcal{M}_n$ as soon as $k\in\mathcal{M}_n$.

There are two typical examples of discrete abelian torsion groups to keep in mind.

\begin{example}\label{ex:discretedyadicgroup}
Let us use the standard notation $\mathbb{Z}_n$ for the cyclic group of residues modulo $n$, i.e., $\mathbb{Z}_n\cong\mathbb{Z}/n\mathbb{Z}$. Consider the infinite direct sum of cyclic groups
\begin{align}
\mathcal{L} & := \bigoplus_{i\in\mathbb{N}}\mathbb{Z}_{n_i} = \mathbb{Z}_{n_1} \oplus \mathbb{Z}_{n_2} \oplus \mathbb{Z}_{n_3} \oplus \cdots \label{eq:sumofcyclicgroups} \\
& \,= \big\{ (a_j)_{j=1}^{\infty}\in{\textstyle\prod_{j=1}^{\infty}\mathbb{Z}_{n_j}} : \text{only finitely many }a_j\text{ are nonzero} \big\}, \nonumber
\end{align}
where $n_1,n_2,n_3,\ldots$ is an arbitrary sequence of positive integers. By \cite[Prop.~4.8]{F95} and \cite[Thm.~4.31]{F95} its dual group $\widehat{\mathcal{L}}$ is isomorphic to $\prod_{i\in\mathbb{N}}\mathbb{Z}_{n_i}$, which are the so-called \emph{Vilenkin groups}.
A particular case is
\[ \mathbb{Z}_2^\omega = \mathbb{Z}_{2} \oplus \mathbb{Z}_{2} \oplus \mathbb{Z}_{2} \oplus \cdots, \]
which we simply call the \emph{discrete dyadic group} and its dual is known as the \emph{Cantor dyadic group}.

The group $\mathbb{Z}_2^\omega$ is bijectively mapped to the set of nonnegative integers $\mathbb{N}_0$ via
\[ \mathbb{Z}_2^\omega\to\mathbb{N}_0,\quad (b_j)_{j=1}^{\infty}\mapsto\sum_{j=1}^{\infty} b_j 2^{j-1} \]
and its inverse is given by
\[ \mathbb{N}_0\to\mathbb{Z}_2^\omega,\quad b\mapsto\text{the sequence of binary digits of $b$ viewed from right to left}. \]
Transporting the group operation from $\mathbb{Z}_2^\omega$ to $\mathbb{N}_0$ via the above bijection yields a binary operation called the \emph{nim-addition}, which plays an important role in the theory of impartial combinatorial games; see \cite[Ch.~6]{C01}.
It is often convenient to represent the discrete dyadic group on the set of nonnegative integers, whenever we need to list some of its elements, as we do in Table~\ref{tab:examplesofvectors2} below.

It is also natural to regard $\mathcal{L}=\mathbb{Z}_2^\omega$ as a lattice in the so-called Walsh field $G=\mathbb{W}$, consisting of two-sided sequences of binary digits $(b_j)_{j\in\mathbb{Z}}$ such that $b_j=0$ when $j$ is large enough, but which can have infinitely many nonzero digits for negative indices $j$.
In harmonic analysis $\mathbb{W}$ is commonly considered as a ``toy model'' for the (nonnegative) reals and $\mathbb{Z}_2^\omega$ is then regarded as a toy model for the (nonnegative) integers. Many difficult problems in analysis and combinatorics can be first studied in those simplified models; see \cite[Ch.~8]{T06} and \cite[Sec.~5]{G05} respectively.

If $\sup_{j}n_j<\infty$ and $N$ is the least common multiple of the numbers $n_1,n_2,\ldots$, then a notable property of \eqref{eq:sumofcyclicgroups} is $Nk=0$ for each $k\in\mathcal{L}$. We say that $\mathcal{L}$ has a \emph{bounded exponent} and the smallest such positive integer $N$ is called the \emph{exponent} of $\mathcal{L}$.
Conversely, by the first Pr\"{u}fer theorem (see \cite[Ch.~3]{F70}) any countable abelian group of bounded exponent is isomorphic to \eqref{eq:sumofcyclicgroups} for some bounded sequence of positive integers $(n_j)_{j=1}^{\infty}$.
\end{example}

\begin{example}\label{ex:prufergroup}
Take a prime number $p$ and consider the set
\[ \mathcal{L} := \Big\{ \frac{m}{p^n} \,:\, n\in\mathbb{N}_0,\, m\in\mathbb{Z},\, 0\leq m<p^n \Big\} \]
with the binary operation being the addition modulo $1$. This group is called the \emph{Pr\"{u}fer $p$-group} and it is sometimes denoted $\mathbb{Z}(p^\infty)$. The dual group $\widehat{\mathcal{L}}$ is isomorphic to the group of $p$-adic integers, see \cite[Thm.~4.12]{F95}, so it is important in number theory. The only finite subgroups of $\mathcal{L}$ are
\[ \mathcal{M}_n := \Big\{ \frac{m}{p^n} \,:\, m\in\mathbb{Z},\, 0\leq m<p^n \Big\} \]
for each nonnegative integer $n$ and observe that $\mathcal{M}_0\subseteq\mathcal{M}_1\subseteq\mathcal{M}_2\subseteq\cdots$ and $\mathcal{M}_n \cong \mathbb{Z}_{p^n}$.
It can be said that the Pr\"{u}fer group is a direct limit of the sequence $(\mathbb{Z}_{p^n})_{n=1}^{\infty}$. Actually, the only subgroups of $\mathcal{L}$ are $\mathcal{M}_0,\mathcal{M}_1,\mathcal{M}_2,\ldots$ and the whole group itself, which complements the previous example, where the lattice has abundance of finite subgroups.
\end{example}


\subsection{Redundancy of convex combinations}

Let us begin with the following easy observation. Suppose that we want to be sparing and verify condition \eqref{eq:basiccharppsi} from Theorem~\ref{thm:abstractcharppsi} by taking infimum only over a certain collection $\mathcal{G}$ of finite sets $\Gamma\subset\mathcal{L}$.
Whenever we have $\Gamma_0,\Gamma_1,\ldots,\Gamma_m\in\mathcal{G}$ such that $g_{\Gamma_0}$ is a convex combination of functions $g_{\Gamma_1},\ldots,g_{\Gamma_m}$, then we can freely throw out the set $\Gamma_0$ from the collection $\mathcal{G}$. Indeed, if $g_{\Gamma_0} = \sum_{i=1}^{m} \gamma_i g_{\Gamma_i}$ for some nonnegative numbers $\gamma_1,\ldots,\gamma_m$ adding up to $1$, then obviously
\[ \int_{\widehat{\mathcal{L}}} g_{\Gamma_0}(\tau) p_\psi(\tau) d\tau
= \sum_{i=1}^{m} \gamma_i \int_{\widehat{\mathcal{L}}} g_{\Gamma_i}(\tau) p_\psi(\tau) d\tau
\geq \min_{1\leq i\leq m} \int_{\widehat{\mathcal{L}}} g_{\Gamma_i}(\tau) p_\psi(\tau) d\tau. \]

It might be easier to spot dependencies among the functions $g_{\Gamma}$ by looking at ``sequences'' of their Fourier coefficients $v_{\Gamma}\in\ell^{\infty}(\mathcal{L})$. By part (a) of Lemma~\ref{lemma:gdefprop} we know that $v_\Gamma\colon\mathcal{L}\to\mathbb{C}$ is given by
\[ v_\Gamma(k) = \frac{\mathop{\textup{card}}(\Gamma\cap(k+\Gamma))}{\mathop{\textup{card}}\Gamma} \]
and we can view $v_\Gamma$ as an infinite vector of numbers from $[0,1]$ such that all but finitely many if its entries are equal to $0$.
For each finite set $\Gamma\subset\mathcal{L}$ we write $\mathbbm{1}_\Gamma$ for the characteristic function of the set $\Gamma$, which can also be thought of as an infinite vector with only finitely many nonzero elements.
Several examples for the groups $\mathcal{L}=\mathbb{Z}$ and $\mathcal{L}=\mathbb{Z}_2^\omega$ are given in Tables~\ref{tab:examplesofvectors1} and \ref{tab:examplesofvectors2} respectively.
For example, on the latter group we have
\begin{align*}
v_{\{0,1,2\}} & = \textstyle\frac{1}{6}v_{\{0,1\}} + \frac{1}{6}v_{\{0,2\}} + \frac{1}{6}v_{\{0,3\}} + \frac{1}{2}v_{\{0,1,2,3\}} \\
& = \textstyle\frac{1}{3}v_{\{0\}} + \frac{2}{3}v_{\{0,1,2,3\}} .
\end{align*}

\begin{table}
\begin{center}\begin{tabular}{c|c|c|}
$\Gamma$ & $\mathbbm{1}_\Gamma$ & $v_\Gamma$ \\ \hline
$\{0\}$ & $(\ldots,0,0,0,1,0,0,0,\ldots)$ & $(\ldots,0,0,0,1,0,0,0,\ldots)$ \\ \hline
$\{0,1\}$ & $(\ldots,0,0,0,1,1,0,0,\ldots)$ & $(\ldots,0,0,\frac{1}{2},1,\frac{1}{2},0,0,\ldots)$ \\ \hline
$\{0,1,2\}$ & $(\ldots,0,0,0,1,1,1,0,\ldots)$ & $(\ldots,0,\frac{1}{3},\frac{2}{3},1,\frac{2}{3},\frac{1}{3},0,\ldots)$ \\ \hline
$\{0,2\}$ & $(\ldots,0,0,0,1,0,1,0,\ldots)$ & $(\ldots,0,\frac{1}{2},0,1,0,\frac{1}{2},0,\ldots)$ \\ \hline
$\{0,1,2,3\}$ & $(\ldots,0,0,0,1,1,1,1,\ldots)$ & $(\ldots,\frac{1}{4},\frac{1}{2},\frac{3}{4},1,\frac{3}{4},\frac{1}{2},\frac{1}{4},\ldots)$ \\ \hline
$\{0,3\}$ & $(\ldots,0,0,0,1,0,0,1,\ldots)$ & $(\ldots,\frac{1}{2},0,0,1,0,0,\frac{1}{2},\ldots)$ \\ \hline
$\{0,1,3\}$ & $(\ldots,0,0,0,1,1,0,1,\ldots)$ & $(\ldots,\frac{1}{3},\frac{1}{3},\frac{1}{3},1,\frac{1}{3},\frac{1}{3},\frac{1}{3},\ldots)$ \\ \hline
\end{tabular}\end{center}
\caption{A table of vectors $\mathbbm{1}_\Gamma$ and $v_\Gamma$ for several sets $\Gamma\subset\mathbb{Z}$.
Dots on either side replace sequences of zeros.}
\label{tab:examplesofvectors1}
\end{table}

\begin{table}
\begin{center}\begin{tabular}{c|c|c|}
$\Gamma$ & $\mathbbm{1}_\Gamma$ & $v_\Gamma$ \\ \hline
$\{0\}^\ast$ & $(1,0,0,0,\ldots)$ & $(1,0,0,0,\ldots)$ \\ \hline
$\{0,1\}^\ast$ & $(1,1,0,0,\ldots)$ & $(1,1,0,0,\ldots)$ \\ \hline
$\{0,2\}^\ast$ & $(1,0,1,0,\ldots)$ & $(1,0,1,0,\ldots)$ \\ \hline
$\{0,3\}^\ast$ & $(1,0,0,1,\ldots)$ & $(1,0,0,1,\ldots)$ \\ \hline
$\{0,1,2\}$ & $(1,1,1,0,\ldots)$ & \,\,$(1,\frac{2}{3},\frac{2}{3},\frac{2}{3},\ldots)$ \\ \hline
$\{0,1,3\}$ & $(1,1,0,1,\ldots)$ & \,\,$(1,\frac{2}{3},\frac{2}{3},\frac{2}{3},\ldots)$ \\ \hline
$\{0,1,2,3\}^\ast$ & $(1,1,1,1,\ldots)$ & $(1,1,1,1,\ldots)$ \\ \hline
\end{tabular}\end{center}
\caption{A table of vectors $\mathbbm{1}_\Gamma$ and $v_\Gamma$ for several sets $\Gamma\subset\mathbb{Z}_2^\omega\cong\mathbb{N}_0$.
Subgroups are marked with an asterisk.}
\label{tab:examplesofvectors2}
\end{table}

The Fourier transform $\ell^2(\mathcal{L})\to\textup{L}^2(\widehat{\mathcal{L}})$ is a linear bijection, so any convex dependence among the vectors $v_\Gamma$ translates into convex dependence among the functions $g_\Gamma$ and vice versa.
Certain lattices might not even have many such dependencies. This is why we restrict our attention to torsion groups throughout this section.

Here is an easy consequence of Theorem~\ref{thm:abstractcharppsi} and the previous observations.
Recall that an \emph{extreme point} of a convex set is any point that cannot be expressed as a nontrivial convex combination of two different points in that set.

\begin{corollary}\label{cor:charpsitorsion}
Suppose that $\mathcal{L}$ is a torsion lattice. Let $\mathcal{K}\subseteq\textup{L}^1(\widehat{\mathcal{L}})$ be the convex hull of the functions $g_\Gamma$ as $\Gamma$ ranges over all nonempty finite subsets of $\mathcal{L}$ and let $\mathcal{E}$ be the collection of functions $g_\Gamma$ that are also extreme points of $\mathcal{K}$.
The system $\mathcal{F}_\psi$ is democratic if and only if $p_\psi\in\textup{L}^\infty(\widehat{\mathcal{L}})$ and
\begin{equation}\label{eq:basiccharppsi2}
\inf_{g\in\mathcal{E}} \int_{\widehat{\mathcal{L}}} g(\tau) p_\psi(\tau) d\tau > 0.
\end{equation}
\end{corollary}

\begin{proof}[Proof of Corollary~\ref{cor:charpsitorsion}]
Modulo Theorem~\ref{thm:abstractcharppsi} the necessity of obvious, while for the sufficiency we only need to show that condition \eqref{eq:basiccharppsi2} implies condition \eqref{eq:basiccharppsi}. Taking into account the remarks from the beginning of this subsection, it remains to prove that each function $g_\Gamma$ is a convex combination of the functions from $\mathcal{E}$.

Take an arbitrary nonempty finite $\Gamma_0\subset\mathcal{L}$.
Choose a finite subgroup $\mathcal{M}$ of $\mathcal{L}$ that contains $\Gamma_0-\Gamma_0$; it could simply be a member of the sequence \eqref{eq:sequencesubgroups}.
Let $\mathcal{Q}$ be the set of all generalized trigonometric polynomials $q$ of the form
\[ q(\tau) = \sum_{k\in\mathcal{M}} \alpha_k e(k,\tau) \]
for some coefficients $\alpha_k\in\mathbb{R}$. Furthermore, let $\mathcal{K}'\subseteq\textup{L}^1(\widehat{\mathcal{L}})$ be the convex hull of the functions $g_\Gamma$ that also belong to the set $\mathcal{Q}$.

At first we claim that $\mathcal{K}'=\mathcal{K}\cap\mathcal{Q}$. The inclusion ``$\subseteq$'' is obvious as $\mathcal{Q}$ is a convex set, so we only need to show the reverse inclusion ``$\supseteq$''. Any $q\in\mathcal{K}\cap\mathcal{Q}$ is (by the definition of $\mathcal{K}$) a convex combination of some functions $g_{\Gamma_1},\ldots,g_{\Gamma_m}$, i.e., there exist $\gamma_1,\ldots,\gamma_m>0$ such that $\sum_{i=1}^{m}\gamma_i=1$ and $q = \sum_{i=1}^{m} \gamma_i g_{\Gamma_i}$. If we had $g_{\Gamma_j}\not\in\mathcal{Q}$ for some index $j$, then there would exist $k\in\mathcal{L}\setminus\mathcal{M}$ such that $v_{\Gamma_j}(k)>0$. This would imply that the $k$-th Fourier coefficient of $q$ is $\sum_{i=1}^{m} \gamma_i v_{\Gamma_i}(k) > 0$ and contradict $q\in\mathcal{Q}$. Therefore, $g_{\Gamma_i}\in\mathcal{Q}$ for $i=1,\ldots,m$, which precisely means $q\in\mathcal{K}'$.
Since by the construction $g_{\Gamma_0}\in\mathcal{K}\cap\mathcal{Q}$, in particular we conclude $g_{\Gamma_0}\in\mathcal{K}'$.

Observe that $\mathcal{Q}$ is a finite-dimensional real vector space; it is isomorphic to $\mathbb{R}^{\mathop{\textup{card}}\mathcal{M}}$.
Moreover, there are only finitely many different function $g_\Gamma$ belonging to $\mathcal{Q}$.
Indeed, take one corresponding $\Gamma$. By translating it we may assume that $0\in\Gamma$; the resulting function $g_{\Gamma}$ will not change.
Then we must have $\Gamma\subseteq\mathcal{M}$, since existence of $k\in\Gamma\setminus\mathcal{M}$ would imply
\[ v_\Gamma(k) = \frac{\mathop{\textup{card}}(\Gamma\cap(k+\Gamma))}{\mathop{\textup{card}}\Gamma} \geq \frac{\mathop{\textup{card}}(\{k\})}{\mathop{\textup{card}}\Gamma} > 0, \]
so $g_\Gamma$ could not belong to $\mathcal{Q}$.
Using Minkowski's theorem (i.e., a finite-dimensional variant of the Krein-Milman theorem) we conclude that each point of $\mathcal{K}'$ is a convex combination of some functions $g_\Gamma$ that are also extreme points of $\mathcal{K}'$. In particular this holds for $g_{\Gamma_0}$.

Finally, it remains to show that each function $g_{\Gamma}$ which is an extreme point of $\mathcal{K}'$ is also an extreme point of the larger convex set $\mathcal{K}$. If the latter was not true, then we could write $g_\Gamma$ as a convex combination $\sum_{i=1}^{m} \gamma_i g_{\Gamma_i}$ with $m\geq 2$ and $\gamma_i>0$ for each index $i$. As before, from $g_{\Gamma}\in\mathcal{Q}$ we would conclude $g_{\Gamma_i}\in\mathcal{Q}$ for $i=1,\ldots,m$, so that indeed $g_{\Gamma_i}\in\mathcal{K}'$ for each $i$, which would contradict the fact that $g_{\Gamma}$ is an extreme point of  $\mathcal{K}'$.
\end{proof}

If $\Gamma_0\subseteq\mathcal{M}_N$ for some large enough integer $N$, where $(\mathcal{M}_n)_{n=0}^{\infty}$ is a sequence from \eqref{eq:sequencesubgroups}, then from the previous proof we also know that $g_{\Gamma_0}$ is an extreme point of $\mathcal{K}$ if and only if $v_{\Gamma_0}$ is an extreme point of the convex hull of the points $v_{\Gamma}$ as $\Gamma$ ranges over nonempty subsets of $\mathcal{M}_N$.

When $\mathcal{M}$ is a subgroup of $\mathcal{L}$, then we reserve the notation $\mathcal{M}^\perp$ for the orthogonal complement of $\mathcal{M}$ relative to $\mathcal{L}$, i.e.,
\[ \mathcal{M}^\perp := \{\tau\in\widehat{\mathcal{L}} \,:\, e(k,\tau)=1 \text{ for each } k\in \mathcal{M}\}. \]
Since the topology of $\mathcal{L}$ is discrete, $\mathcal{M}$ is automatically closed in $\mathcal{L}$, so $\mathcal{M}\mapsto\mathcal{M}^\perp$ is now a bijective correspondence between subgroups of $\mathcal{L}$ and closed subgroups of $\widehat{\mathcal{L}}$. Moreover, suppose that $\mathcal{M}$ is finite. From the isomorphism \eqref{eq:groupiso2} we get
\begin{equation}\label{eq:groupiso2m}
\widehat{\mathcal{L}}/\mathcal{M}^\perp\cong\widehat{\mathcal{M}}
\end{equation}
and, since the dual group of a finite abelian group is isomorphic to itself (see \cite[Cor.~4.7]{F95}), we conclude
\begin{equation}\label{eq:groupiso2mcard}
\mathop{\textup{card}}(\widehat{\mathcal{L}}/\mathcal{M}^\perp) = \mathop{\textup{card}}\mathcal{M}.
\end{equation}
Consequently, $\mathcal{M}^\perp$ is a closed subgroup of $\widehat{\mathcal{L}}$ of finite index and each such subgroup must also be open, since its complement is a union of finitely many closed cosets. Conversely, every closed subgroup of $\widehat{\mathcal{L}}$ is certainly of the form $\mathcal{M}^\perp$ for some subgroup $\mathcal{M}$ of $\mathcal{L}$. If it is also of finite index, then from the same isomorphism \eqref{eq:groupiso2m} we conclude that $\mathcal{M}$ must be finite. Let us summarize by saying that $\mathcal{M}\mapsto\mathcal{M}^\perp$ is also a bijective correspondence between finite subgroups of $\mathcal{L}$ and closed subgroups of $\widehat{\mathcal{L}}$ of finite index (which are automatically also open).

Take a torsion lattice $\mathcal{L}$. Orthogonal complements of the terms in \eqref{eq:sequencesubgroups} satisfy
\[ \mathcal{L} = \mathcal{M}_0^\perp \supset \mathcal{M}_1^\perp \supset \mathcal{M}_2^\perp \supset\cdots \quad\text{and}\quad
\bigcap_{n=0}^{\infty} \mathcal{M}_n = \{0\}. \]
Suppose that we have the periodization function $p_\psi$ associated with some square-integrable function $\psi$ for which we want to verify the democratic property of $\mathcal{F}_\psi$. In any real-world application it is likely that $p_\psi$ will be given to us with a precision to certain scale and we are required to verify conditions $p_\psi\leq C$ and $\int_{\widehat{\mathcal{L}}} g_\Gamma(\tau) p_\psi(\tau) d\tau\geq c$, where $C$ (resp.\@ $c$) is some reasonably large (resp.\@ small) constant. If we only have information about the averages of $p_\psi$ on the cosets of $\mathcal{M}_N^\perp$ (for some large but fixed positive integer $N$), then it only makes sense to test the lower bound for $\Gamma\subseteq\mathcal{M}_N$, as $g_\Gamma$ is constant on the cosets of $\mathcal{M}_N^\perp$ only for such sets $\Gamma$. This allows us to preprocess the finite group $\mathcal{M}_N$ in the search for extreme points. Numerical experimentation in Example~\ref{ex:discretedyadicgroup2} suggests that this can save the computational time significantly.


\subsection{Subgroups as extreme points and some examples}

We continue by investigating a special role of finite subgroups of $\mathcal{L}$. It is already hinted by the following auxiliary lemma.

\begin{lemma}\ \label{lemma:gbasic}
\begin{itemize}
\item[(a)]
If $\Gamma=m+\mathcal{M}$, where $m\in\mathcal{L}$ and $\mathcal{M}$ is a finite subgroup of $\mathcal{L}$, then
\[ g_\Gamma(\tau) = \sum_{k\in\mathcal{M}} e(k,\tau); \quad \tau\in\widehat{\mathcal{L}} \]
and
\[ g_\Gamma(\tau) = \begin{cases} 1/|\mathcal{M}^\perp| & \text{for } \tau\in\mathcal{M}^\perp, \\ 0 & \text{for } \tau\not\in\mathcal{M}^\perp. \end{cases} \]
\item[(b)]
Conversely, if $\Gamma\subseteq\mathcal{L}$ is a nonempty finite set such that all coefficients of $g_\Gamma$ are either $0$ or $1$, i.e.,
\[ g_\Gamma(\tau) = \sum_{k\in S} e(k,\tau) \]
for some finite $S\subseteq\mathcal{L}$, then $\Gamma$ has to be of the form $\Gamma=m+\mathcal{M}$ where $m\in\mathcal{L}$ and $\mathcal{M}$ is a finite subgroup of $\mathcal{L}$, and indeed $S=\mathcal{M}$.
\end{itemize}
\end{lemma}

\begin{proof}[Proof of Lemma~\ref{lemma:gbasic}]
(a) The first claim is a consequence of part (a) of Lemma~\ref{lemma:gdefprop}, since the cosets $m+\mathcal{M}$ and $k+m+\mathcal{M}$ are either equal (when $k\in\mathcal{M}$) or disjoint (when $k\not\in\mathcal{M}$). For the second claim observe that from \eqref{eq:groupiso2mcard} we obtain
\begin{equation}\label{eq:basiccardmeas}
\mathop{\textup{card}}\mathcal{M} = |\widehat{\mathcal{L}}|/|\mathcal{M}^\perp| = 1/|\mathcal{M}^\perp|,
\end{equation}
so for $\tau\in\mathcal{M}^\perp$ we have
\[ g_\mathcal{M}(\tau) = \sum_{k\in\mathcal{M}} e(k,\tau) = \sum_{k\in\mathcal{M}} 1 = \mathop{\textup{card}}\mathcal{M} = \frac{1}{|\mathcal{M}^\perp|}. \]
On the other hand, if $\tau\not\in\mathcal{M}^\perp$, then there exists $k_0\in\mathcal{M}$ such that $e(k_0,\tau)\neq 1$. Therefore,
\begin{align*}
g_\mathcal{M}(\tau) & = \sum_{k\in\mathcal{M}} e(k_0+k-k_0,\tau) = e(k_0,\tau) \sum_{k\in\mathcal{M}} e(k-k_0,\tau) \\
& = e(k_0,\tau) \sum_{m\in \mathcal{M}-k_0=\mathcal{M}} e(m,\tau) = \underbrace{e(k_0,\tau)}_{\neq 1}g_\mathcal{M}(\tau),
\end{align*}
which implies $g_\mathcal{M}(\tau)=0$.

(b) Recall that $e(k,\cdot)$; $k\in\mathcal{L}$ form an orthonormal basis for $\textup{L}^2(\widehat{\mathcal{L}})$, so in particular these functions are linearly independent. Combining the assumption with part (a) of Lemma~\ref{lemma:gdefprop} and equaling the coefficients we see that for any $k\in\mathcal{L}$ the sets $\Gamma$ and $k+\Gamma$ are either equal or disjoint. Fix an arbitrary $m\in\Gamma$ and define $\mathcal{M}:=\Gamma-m$. We need to prove that $\mathcal{M}$ is a subgroup of $\mathcal{L}$. Take any $a,b\in\mathcal{M}$ and observe that
\[ m \in \Gamma\cap(m-b+\mathcal{M}) = \Gamma\cap(-b+\Gamma). \]
Since this intersection is nonempty, we must have $\Gamma=-b+\Gamma$ i.e., $\mathcal{M}=-b+\mathcal{M}$, which in particular implies $a-b\in -b+\mathcal{M}=\mathcal{M}$. We have just shown $\mathcal{M}-\mathcal{M}\subseteq\mathcal{M}$, which verifies that $\mathcal{M}$ is a subgroup. Applying part (a) and using linear independence once again we also conclude $S=\mathcal{M}$.
\end{proof}

In particular we see that the vector of Fourier coefficients $v_\mathcal{M}$ of a finite subgroup $\mathcal{M}$ of $\mathcal{L}$ has only $\{0,1\}$-entries. Consequently, the corresponding function $g_\mathcal{M}$ certainly belongs to the set of extreme points $\mathcal{E}$ of the set $\mathcal{K}$ described in Corollary~\ref{cor:charpsitorsion}.
Part (b) of Lemma~\ref{lemma:gbasic} characterizes subgroup cosets as the only nonempty finite sets $\Gamma$ having $\{0,1\}$-Fourier coefficients. However, we need to emphasize that in general these are not the only extreme points; Example~\ref{ex:prufergroup2} will indirectly disprove that fact on the Pr\"{u}fer $2$-group.
On a related note, polytopes whose vertices are tuples with all entries from $\{0,1\}$ are called \emph{$0/1$-polytopes}. They are combinatorially interesting and extensively studied; see \cite{Z00}.

When $\Gamma=\mathcal{M}$ is a finite subgroup of $\mathcal{L}$, then from part (a) of the previous lemma we see that the integral in \eqref{eq:basiccharppsi} becomes a density-type expression for the periodization:
\[ \frac{1}{|\mathcal{M}^\perp|} \int_{\mathcal{M}^\perp} p_\psi(\tau) d\tau. \]

\begin{example}\label{ex:discretedyadicgroup2}
This is a continuation of Example~\ref{ex:discretedyadicgroup}; recall the discrete dyadic group $\mathbb{Z}_2^\omega$ introduced there. Take
\[ \mathcal{M}_n := \mathbb{Z}_2 \oplus \cdots \oplus \mathbb{Z}_2 \oplus \{0\} \oplus \{0\} \oplus \cdots \cong \mathbb{Z}_2^n \]
as finite subgroups that exhaust the whole group.
Each $\mathcal{M}_n$ is actually an $n$-dimensional vector space over $\mathbb{Z}_2$. The number of $k$-dimensional subspaces of $\mathcal{M}_n$, $0\leq k\leq n$, is given by the particular case $q=2$ of the $q$-binomial coefficient,
\[ {n \choose k}_q := \frac{(q^n-1)(q^{n-1}-1)\cdots(q^{n-k+1}-1)}{(q^1-1)(q^2-1)\cdots(q^k-1)}. \]
Consequently, cardinality of the set $\mathcal{E}_n$ of extreme points $g_\Gamma\in\mathcal{E}$ such that $\Gamma\subseteq\mathcal{M}_n$ is at least $\sum_{k=0}^{n}{n \choose k}_2$. The last expression defines sequence A006116 in the encyclopedia OEIS \cite{OEIS} (beginning with $1$, $2$, $5$, $16$, $67$, \ldots) and its asymptotic behavior is well-known. We can find
\[ \liminf_{n\to\infty}\frac{\mathop{\textup{card}}(\mathcal{E}_n)}{2^{n^2/4}} >0, \]
so the number of extreme points in $\mathcal{M}_n$ grows super-exponentially in $n$.
The actual numerical data are given in Table~\ref{tab:growthofpoints}.
A lot of torsion already causes that many functions $g_\Gamma$ coincide, but it is expected that removing convex combinations reduces that number further significantly.
\end{example}

\begin{table}
\begin{center}\begin{tabular}{l|r|r|r|r|r|}
& $\mathbb{Z}_2^0$ & $\mathbb{Z}_2^1$ & $\mathbb{Z}_2^2$ & $\mathbb{Z}_2^3$ & $\mathbb{Z}_2^4$ \\ \hline
Total number of subsets $\Gamma$ & $1$ & $3$ & $15$ & $255$ & $65535$ \\ \hline
Number of different points $g_\Gamma$ & $1$ & $2$ & $6$ & $45$ & $3966$ \\ \hline
Number of extreme points $g_\Gamma$ & $1$ & $2$ & $5$ & $16$ & $\geq67$ \\ \hline
\end{tabular}\end{center}
\caption{Numerical data for the discrete dyadic group.}
\label{tab:growthofpoints}
\end{table}

The following lemma will be needed in the next example.

\begin{lemma}\label{lemma:gonsubgroups}
If $\Gamma\subset\mathcal{L}$ is a nonempty finite set and $\mathcal{M}$ is a finite subgroup of $\mathcal{L}$, then
\[ \int_{\mathcal{M}^\perp}\! g_\Gamma(\tau) d\tau = \frac{\mathop{\textup{card}} \{(k,k')\in\Gamma\times\Gamma \,:\, k-k'\in\mathcal{M}\}}{\mathop{\textup{card}}\mathcal{M} \mathop{\textup{card}}\Gamma}. \]
\end{lemma}

\begin{proof}[Proof of Lemma~\ref{lemma:gonsubgroups}]
Let us begin by showing
\[ \int_{\mathcal{M}^\perp}\! e(k,\tau) d\tau = \begin{cases} |\mathcal{M}^\perp| & \text{for } k\in\mathcal{M}, \\ 0 & \text{for } k\not\in\mathcal{M}. \end{cases} \]
This claim is obvious for $k\in\mathcal{M}$, so take $k\in\mathcal{L}\setminus\mathcal{M}$. By $(\mathcal{M}^\perp)^\perp=\mathcal{M}$ there exists $\tau_0\in\mathcal{M}^\perp$ such that $e(k,\tau_0)\neq 1$. Now we can write (using translation invariance of the Haar measure on $\widehat{\mathcal{L}}$):
\begin{align*}
\int_{\mathcal{M}^\perp}\! e(k,\tau) d\tau & = \int_{\mathcal{M}^\perp}\! e(k,\tau_0+\tau-\tau_0) d\tau = e(k,\tau_0) \int_{\mathcal{M}^\perp}\! e(k,\tau-\tau_0) d\tau \\
& = e(k,\tau_0) \int_{-\tau_0+\mathcal{M}^\perp}\! e(k,\tau) d\tau = \underbrace{e(k,\tau_0)}_{\neq 1} \int_{\mathcal{M}^\perp}\! e(k,\tau) d\tau,
\end{align*}
which implies that the above integral is $0$, as needed. Applying part (a) of Lemma~\ref{lemma:gdefprop}, integrating term-by-term, and using \eqref{eq:basiccardmeas} we get
\[ \int_{\mathcal{M}^\perp}\! g_\Gamma(\tau) d\tau = \frac{\sum_{k\in\mathcal{M}}\mathop{\textup{card}}(\Gamma\cap(k+\Gamma))}{\mathop{\textup{card}}\mathcal{M} \mathop{\textup{card}}\Gamma}. \]
In order to transform this formula into the desired one it remains to observe that the numerator above equals
\[ \mathop{\textup{card}}\big\{(k_1,k_2)\in\Gamma\times\Gamma \,:\, k_1-k_2\in\mathcal{M}\big\}, \]
which is easily seen by double counting.
\end{proof}

One might get an impression that, in the case of torsion lattices, it is enough to test the democratic property on subgroups. However, this is not the case, as the following example shows.

\begin{example}\label{ex:prufergroup2}
This is a continuation of Example~\ref{ex:prufergroup}; recall the Pr\"{u}fer $2$-group $\mathbb{Z}(2^\infty)$ and its subgroups $\mathcal{M}_n \cong \mathbb{Z}_{2^n}$.
For each positive integer $n$ choose $s_n\in\mathcal{M}_n\setminus\mathcal{M}_{n-1}$. Then $\mathcal{M}_{n-1}$ and $s_n+\mathcal{M}_{n-1}$ are the only two cosets of the smaller subgroup in the larger one. For any positive integer $n$ define
\[ \Gamma_n := \{0,s_1\} + \{0,s_3\} + \cdots + \{0,s_{2n-1}\}. \]
Any $k,k'\in\Gamma_n$ have unique representations as $k=\sum_{j=1}^{n}\alpha_j s_{2j-1}$, $k'=\sum_{j=1}^{n}\alpha'_j s_{2j-1}$, where $\alpha_j,\alpha'_j\in\{0,1\}$ for each index $j=1,\ldots,n$. For a fixed integer $0\leq m\leq 2n-1$ we observe that
\[ k-k'\in\mathcal{M}_m \Longleftrightarrow \alpha_j=\alpha'_j \text{ for all indices $j$ such that } j>(m+1)/2. \]
Consequently,
\[ \mathop{\textup{card}}\big\{(k,k')\in\Gamma_n\times\Gamma_n \,:\, k-k'\in\mathcal{M}_m\big\}
= \begin{cases}
2^{n+m/2} & \text{if } 0\leq m\leq 2n-1 \text{ is even}, \\
2^{n+(m+1)/2} & \text{if } 0\leq m\leq 2n-1 \text{ is odd},
\end{cases} \]
so Lemma~\ref{lemma:gonsubgroups} gives
\[ \int_{\mathcal{M}_m^\perp}\! g_{\Gamma_n}(\tau) d\tau
= \begin{cases}
2^{-m/2} & \text{if } 0\leq m\leq 2n-1 \text{ is even}, \\
2^{-(m-1)/2} & \text{if } 0\leq m\leq 2n-1 \text{ is odd},
\end{cases} \]
From this we conclude
\begin{equation}\label{eq:counterexg}
\int_{\mathcal{M}_{2i-1}^\perp\setminus\mathcal{M}_{2i}^\perp}\! g_{\Gamma_n}(\tau) d\tau = 2^{-i}
\end{equation}
for $i=1,2,\ldots,n-1$.

Let us now choose a square-integrable function $\psi$ such the periodization function $p_\psi$ is equal to the characteristic function of the set $\bigcup_{i=0}^{\infty}(\mathcal{M}_{2i}^\perp\setminus\mathcal{M}_{2i+1}^\perp)$.
(For this purpose one can simply take $G$ to also equal $\mathbb{Z}(2^\infty)$.)
Using \eqref{eq:counterexg} we get
\[ \int_{\widehat{\mathcal{L}}} g_{\Gamma_n}(\tau) p_\psi(\tau) d\tau
= 1 - \sum_{i=1}^{\infty} \int_{\mathcal{M}_{2i-1}^\perp\setminus\mathcal{M}_{2i}^\perp} g_{\Gamma_n}(\tau) d\tau
\leq 1 - \sum_{i=1}^{n-1} 2^{-i} = 2^{-n+1}, \]
so by taking $n\to\infty$ we see that condition \eqref{eq:basiccharppsi} fails and $\mathcal{F}_\psi$ cannot be a democratic system.
On the other hand, for each nonnegative integer $n$ by part (a) of Lemma~\ref{lemma:gbasic} we have
\[ \int_{\widehat{\mathcal{L}}} g_{\mathcal{M}_n}(\tau) p_\psi(\tau) d\tau
= \frac{1}{|\mathcal{M}_n^\perp|} \int_{\mathcal{M}_n^\perp} p_\psi(\tau) d\tau
= \frac{\sum_{i\geq n/2}|\mathcal{M}_{2i}^\perp\setminus\mathcal{M}_{2i+1}^\perp|}{|\mathcal{M}_n^\perp|}
= \begin{cases}
2/3 & \text{if $n$ is even}, \\
1/3 & \text{if $n$ is odd},
\end{cases} \]
so the above quantities, obtained by testing \eqref{eq:basiccharppsi} on subgroups only, are bounded from below by $1/3$.
\end{example}

The previous example shows that, in general, testing the democratic property on subgroups is not sufficient.


\subsection{Closing remarks}
Recall that the paper \cite{HNSS13} conjectures the sufficiency of testing the democratic property on the finite arithmetic progressions in $\mathbb{Z}$.
When we pass to the torsion lattice $\mathcal{L}$, sufficiently long progressions automatically become finite subgroups. In Example~\ref{ex:prufergroup2} we saw that finite subgroups are not enough, but it might still be sufficient to test condition \eqref{eq:basiccharppsi} for the democratic property of $\mathcal{F}_\psi$ by taking only finite sets $\Gamma$ that are approximate subgroups in an appropriate sense. The notion of an \emph{approximate subgroup} was defined in several possible ways in the book \cite{TV06}.

When we are given a concrete sequence of exhausting subgroups \eqref{eq:sequencesubgroups}, numerical data suggest that the number of extreme points $g_\Gamma\in\mathcal{E}$ coming from $\Gamma\subseteq\mathcal{M}_n$ grows at most like $e^{P(\log\mathop{\textup{card}}(\mathcal{M}_n))}$ for some polynomial $P$. Indeed, Example~\ref{ex:discretedyadicgroup2} gives a lower bound of the form $c' e^{c(\log\mathop{\textup{card}}(\mathcal{M}_n))^2}$ for the discrete dyadic group $\mathbb{Z}_2^\omega$, but we were not able to establish the upper bound, either for $\mathbb{Z}_2^\omega$, or for any other torsion lattice $\mathcal{L}$.

We conclude that characterizations of democratic systems of translates still remain without definite answers and we hope that they might attract researchers from various fields.


\end{document}